\documentclass{amsart}[35pt]
\usepackage{amsmath, amsthm, amscd, amssymb, amsfonts, inputenc, tikz}
\UseRawInputEncoding

\usetikzlibrary{shapes,arrows}

\usepackage[margin=3cm]{geometry}

\newtheorem{theorem}{Theorem}[section]
\newtheorem{lem}[theorem]{Lemma}

\newtheorem{cor}[theorem]{Corollary}
\theoremstyle{definition}

\theoremstyle{remark}
\newtheorem{rem}[theorem]{Remark}
\numberwithin{equation}{section}
\begin{document}

\title[Some characterizations of the floor, ceiling and fractional part functions]{Some
characterizations of the floor, ceiling and fractional part functions with related results}

\author[M.H. Hooshmand]{M.H. Hooshmand }

\address{Department of mathematics, Shiraz Branch, Islamic Azad University, Shiraz, Iran}

\email{\tt hadi.hooshmand@gmail.com , MH.Hooshmand@iau.ac.ir}


\date{}
\begin{abstract}
One of the most important issues for the frequent special functions is the uniqueness conditions of such functions.
As far as we know, there are no characterizations for the floor, ceiling, and fractional part functions in general (as real functions
 $f:\mathbb{R}\rightarrow \mathbb{R}$).
As integer-valued functions $f:\mathbb{Q}\rightarrow \mathbb{Z}$, we have found a just common
characterization for the floor and ceiling functions.
   Motivated by the topic of decomposer functions and factors of groups, we introduce and prove some characterizations of the
   above-mentioned functions.
    We also state and prove the relevant results and some of their generalities and developments.
\end{abstract}
\maketitle
\section{Introduction and Preliminaries}
\noindent
One of the features of the frequently used special functions is their uniqueness conditions which are obtained by one or more basic properties.
In such characterizations, a functional equation in which the special function is a solution usually plays the main role. 
Examples are the functional equation $f(x+1)=xf(x)$ satisfied by the gamma function $\Gamma(x)$,
 or $f(x)=a^{f(x-1)}$ with the ultra exponential function $\mbox{uxp}_a(x)$ as a solution, etc (see \cite{Art,MH6}).
But in the case of the floor, ceiling, and fractional part functions (\cite{Rub}) such
 a general characterization (as real-valued functions) was not found.
 The only uniqueness conditions we found is for the floor and ceiling functions, which
 are expressed for the special case $f:\mathbb{Q}\rightarrow \mathbb{Z}$ (as an integer-valued function defined on rational numbers)
 in \cite{Eis}.
 That unique conditions states that $f:\mathbb{Q}\rightarrow \mathbb{Z}$ satisfies the conditions
 $f(x+k)=f(x)+k$ and  $f(\frac{1}{n}f(nx))=f(x)$, for all $x\in \mathbb{Q}$, $k\in \mathbb{Z}$, and $n\in \mathbb{Z}_+$ if and only if
$f(x)=\lfloor x\rfloor$ on $\mathbb{Q}$  or $f(x)=\lceil x\rceil$ on $\mathbb{Q}$.
In the last few years of research, we have introduced and studied the decomposer, associative, canceler, parter, and co-periodic (etc.) functions on basic algebraic structures.
Motivated by the topic, we introduce and prove some uniqueness
conditions for the floor, ceiling, and fractional part functions and obtain some related results.

For this purpose, we will first review some concepts and definitions from those general works \cite{MH1,MH4,MH5} for the particular case of
the additive group of real numbers.
\\
  Let $A,B$ be two subsets of $(\mathbb{R},+)$. We call the sum $A+B$ to be direct, and denote it by $A \dot{+} B$
if the representation of every
  element of $A+B$ by $x=a+b$  with $a\in A$, $b\in B$ is unique.
It is worth noting that the following conditions are equivalent:\\
(1) $A+B=A \dot{+} B$;\\
(2) $(A-A)\cap (B-B)=\{0\}$ (where $A-A:=\{a_1-a_2:a_1,a_2\in A\}$);\\
(3)  $a+B\cap a'+B=\emptyset$, for all distinct elements $a,a'$ of $A$;\\
(4) $A+B=\dot{\bigcup}_{a\in A}a+B$ (where $\dot{\bigcup}$ denotes the disjoint union).\\
Hence, $\mathbb{R}=A\dot{+} B$ if and only if $\mathbb{R}=A+B$  and the sum $A+B$ is direct, in this case we call
$A$ (resp. $B$) a factor of $\mathbb{R}$ relative to $B$(resp. $A$). A subset $A$ is called a factor 
if it is a  factor of $\mathbb{R}$ relative to some subset $B$. For example, every subgroup is a factor
(relative to its transversal).
An important example for it is $\mathbb{R}=\mathbb{Z}\dot{+}[0,1)$ which means every real number can be
 uniquely represented as a sum of the integer and fractional parts of it.\\
 Every factorization $\mathbb{R}=A\dot{+} B$ induces
two natural projections onto $A$ and $B$ as real functions $P_{A}:\mathbb{R}\longrightarrow \mathbb{R}$ with
$P_{A}(\mathbb{R})=A$ and $P_{B}:\mathbb{R}\longrightarrow \mathbb{R}$ with
$P_{B}(\mathbb{R})=B$. Putting $f^*:=\iota_\mathbb{R}-f$, for every $f:\mathbb{R}\rightarrow \mathbb{R}$ where $\iota_\mathbb{R}$ is the
identity function in $\mathbb{R}$, we have
$P_{A}^*=P_{B}$.
 For the factorization $\mathbb{R}=\mathbb{Z}\dot{+}[0,1)$ we have $\lfloor\;\rfloor=P_{\mathbb{Z}}$
(the integer part or floor function) and $\{\;\}=P_{[0,1)}$ (the fractional or decimal part function)
which can be considered as different definitions of these special functions.
Also, we have $\lceil\;\rceil=P_{\mathbb{Z}}$ for the factorization $\mathbb{R}=\mathbb{Z}\dot{+}(-1,0]$.
We call $f:\mathbb{R}\rightarrow \mathbb{R}$ a parting projection if $f=P_{A}$ for some
factorization $\mathbb{R}=A\dot{+} B$. In \cite{MH1} it is proved that $f$ is a parting projection if and only if
it satisfies the decomposer equation:
\begin{equation}
f(f^*(x)+f(y))=f(y) \;\;\; :\;  x,y\in \mathbb{R}
\end{equation}
It is important to know that $\mathbb{R}=f^*(\mathbb{R})+f(\mathbb{R})$, for every $f:\mathbb{R}\rightarrow \mathbb{R}$, but
$\mathbb{R}=f^*(\mathbb{R})\dot{+}f(\mathbb{R})$ if and only if $f$ is a decomposer function.\\
Also, $(1.1)$ together with the condition $f^*(\mathbb{R})\leq \mathbb{R}$ (i.e., the $*$-range of $f$ is a subgroup of $(\mathbb{R},+)$)
is equivalent to the strong decomposer functional equation:
\begin{equation}
f(f^*(x)+y)=f(y) \;\;\; :\;  x,y\in \mathbb{R},
\end{equation}
and also the canceler equation $f(f(x)+y)=f(x+y)$. They have closed relationships to the associative equation
$f(f(x+y)+z)=f(f(x)+f(y+z))$, multiplicative symmetric equation  $f(f(x)+y) =f(x+f(y))$, etc.\\
It is worth noting that a subset of real numbers is a factor if and only if it is the range of a parting projection
(equivalently, decomposer function).
\section{Characterizations of the floor, ceiling and fractional functions}
Now we state and prove a uniqueness theorem that gives a characterization of the floor function.
Of course, first, we introduce it without using the concepts and notations mentioned in the previous section and
then we will get many interesting results and corollaries by using those concepts.
 \begin{theorem}[A characterization of the floor function]
 The floor function  $f=\lfloor\;\rfloor :\mathbb{R}\rightarrow \mathbb{R}$ is the only function
 satisfying the conditions:\\
 $($1$)$ $f(x-f(x)+f(y))=f(y)$ for all $x,y\in \mathbb{R}$;\\
 $($2$)$ The range of the function $x-f(x)$ is $[0,1)$;\\
 $($3$)$ $f(0)=0$.
 \end{theorem}
 \begin{proof}
 It is easy to see that the floor function satisfies the conditions. Conversely,
 let the conditions hold and put $g(x):=x-f(x)$. Thus $g(\mathbb{R})=[0,1)$,
 and putting $A:=f(\mathbb{R})$ we have $$\mathbb{R}=A+[0,1)=\bigcup_{a\in A}[a,a+1).$$
 Our claim is $A=\mathbb{Z}$. Let $k\in A\cap \mathbb{Z}$ (e.g., $k=0$),
 then $k+1\in [a_1,a_1+1)$, for some $a_1\in A$, and hence
 $[a_1,a_1+1)\cap [k,k+1)=\emptyset$. Indeed, if the intersection contains an element $t$,
 then  $t=a_1+d_1=k+d_2$, for some $d_1,d_2\in [0,1)$, and so (1), (2) imply that
 $$
 a_1=f(a_1+d_1)=f(k+d_2)=k,
 $$
which contradicts $k+1\in [a_1,a_1+1)$. Therefore, $k+1\leq a_1$ and $a_1\leq k+1<a_1+1$
 (note that the case $a_1+1\leq k$ is impossible, since $k<k+1<a_1+1$) thus $k+1\in a_1\in A$.
Also, $k-1\in [a_2,a_2+1)$, for some $a_2\in A$, thus $a_2\leq k-1$.
If $a_2< k-1$ then putting $M=\frac{k+a_2+1}{2}$ we have $M\in [a_3,a_3+1)$, for some $a_3\in A$,
and $0<k-(a_2+1)$ which implies the interval $[a_3,a_3+1)$ intersects either $[a_2,a_2+1)$
or  $[k,k+1)$ and this makes a contradiction similar to the above. Hence $k-1\in a_2\in A$.
Up to now, we have proven that $\mathbb{Z}\subseteq A$. Also, $A$ can not have any non-integer element
(because every interval $[a,a+1)$ intersects at least an interval  $[k,k+1)$, for some $k\in \mathbb{Z}$).
Therefore, the claim is proved and so $f(x)\in \mathbb{Z}$, for all $x\in \mathbb{R}$, and
$$
f(x)=f(x)+\lfloor g(x)\rfloor=\lfloor f(x)+ g(x)\rfloor=\lfloor x\rfloor.
$$
This completes the proof.
 \end{proof}
\begin{rem}
With due regard to the conditions of the above theorem, it is clear that the third
condition can be reduced to $f(k_0)=k_0$ for some $k_0\in \mathbb{Z}$  or even $f(x_0)\in \mathbb{Z}$  for some
$x_0\in \mathbb{R}$ (i.e., $f(\mathbb{R})\cap \mathbb{Z}\neq\emptyset)$.
But if $f(\mathbb{R})\cap \mathbb{Z}=\emptyset$, then it is no longer valid.
For if $f(x):=\lfloor x-\frac{1}{2}\rfloor+\frac{1}{2}$, then $f$ satisfies
the conditions (1) and (2) but not (3).
To say more generally, by removing each of the conditions (1), (2), or  $f(\mathbb{R})\cap \mathbb{Z}\neq\emptyset$, the uniqueness condition fails.
\end{rem}
As mentioned before, all parting projections make general solution of $(1.1)$.
Now, we obtain the general solution of $(1.1)$ under the condition $f^*(\mathbb{R})=[0,1)$
(i.e., all functions satisfy the conditions (1) and (2) of Theorem 2.1).
 \begin{theorem}
 The general solution of the equation $(1.1)$ with the condition $f^*(\mathbb{R})=[0,1)$ is
  $f(x)=\lfloor x-c\rfloor+c$ for all $($constant$)$ real numbers $c$.
 \end{theorem}
 \begin{proof}
It is easy to check that all functions $f$ with this form fulfill the conditions.
Conversely, let $f$ be a solution of $(1.1)$ with the mentioned condition, and $A:=f(\mathbb{R})$.
We claim that $A=c+\mathbb{Z}$, for some $c\in \mathbb{R}$. The restricted function
$\lceil\;\rceil\arrowvert_A:A\rightarrow \mathbb{Z}$
(restriction of the ceiling function on $A$) is a strictly increasing onto function.
Because, $\mathbb{R}=\bigcup_{a\in A}[a,a+1)$ and $i\in [a,a+1)\cap \mathbb{Z}$ if and only if
$i=\lceil a\rceil$. Also, if $a,a'\in A$ and $a<a'$, then we consider two cases:\\
$a+1<a'$: we have $\lceil a\rceil<a+1<a'\leq \lceil a'\rceil$ thus $\lceil a\rceil<\lceil a'\rceil$.\\
$a'\leq a+1$: we conclude that $a'\in (a,a+1]$ and $[a,a+1)\cap [a',a'+1)=\emptyset$ requires that
$a'=a+1$. Hence $\lceil a\rceil<\lceil a+1\rceil=\lceil a'\rceil$.\\
Therefore, it is an invertible function with the inverse $\lambda: \mathbb{Z}\rightarrow A$.
We claim that $\lambda(i+1)=\lambda(i)+1$ for all $i\in \mathbb{Z}$. Putting $a_i:=\lambda(i)$
we have
$$
a_i\leq \lceil a_i\rceil=i<i+1=\lceil a_{i+1}\rceil< a_{i+1}+1\; ;\; i\in \mathbb{Z}.
$$
Also, $[a_i,a_i+1)\cap [a_{i+1},a_{i+1}+1)=\emptyset$ since $a_i$ and $a_{i+1}$ are distinct elements of $A$. Hence
$a_i+1\leq a_{i+1}$ ($a_{i+1}+1\leq a_i$ is impossible by the above relation) and $\lambda$ is a strictly increasing function (two-sided sequence).
Thus, these facts establish the claim, since all intervals of the form $[a_i,a_i+1)$ cover $\mathbb{R}$ and there is no such interval
between $[a_i,a_i+1)$ and $[a_{i+1},a_{i+1}+1)$.\\
Therefore, putting $c:=a_0$ we obtain $a_i=c+i$, for all integers $i$. Now, for every $x\in \mathbb{R}$, we have
$x\in [a_i,a_i+1)$ for some $i\in \mathbb{Z}$. We claim that $f(x)=a_i$ or equivalently $f(x)\in [a_i,a_i+1)$.
 Because $x-1<f(x)\leq x$, due to the condition,
and so $$a_{i-1}=a_i-1\leq x-1<f(x)\leq x<a_i+1=a_{i+1}.$$
 Thus $a_{i-1}<f(x)<a_{i+1}$ and so $f(x)=a_{i}$. This means
$f(x)=c+i$,  for all $x$ such that $c+i\leq x<c+i+1$, and hence  $f(x)=c+\lfloor x-c\rfloor$.
 \end{proof}
Note that one can obtain the uniqueness conditions for the characterization of the floor function
(Theorem 2.1) from Theorem 2.3, since the range of $f(x)=\lfloor x-c\rfloor+c$ intersects $\mathbb{Z}$ if and only if $c\in \mathbb{Z}$.
\begin{cor}[A characterization of the ceiling function]
The ceiling function  $f=\lceil\;\rceil :\mathbb{R}\rightarrow \mathbb{R}$ is the only function
 satisfying $(1.1)$ with the conditions:\\
 $($i$)$ $f^*(\mathbb{R})=(-1,0]$;\\
 $($ii$)$ $f(\mathbb{R})\cap \mathbb{Z}\neq\emptyset$.\\
 Moreover, all functions satisfying $(1.1)$ and (i) are of the form $f(x)=\lceil x-c\rceil+c$ for all $c\in \mathbb{R}$.
 \end{cor}
 \begin{proof}
This is a direct result of Theorem 2.1, 2.3 and Remark 2.2. Indeed, if $f$
satisfies the above conditions, then $h(x):=-f(-x)$ fulfill all the conditions of Theorem 2.1 (considering $h$ instead of $f$)
and one can get the result.
\end{proof}
\begin{cor}[A characterization of the fractional part function]
The fractional part function  $f=\{\;\} :\mathbb{R}\rightarrow \mathbb{R}$ is the only function
 satisfying $(1.1)$ with the conditions:\\
 $($i$)$ $f(\mathbb{R})=[0,1)$;\\
$($ii$)$ $f^*(\mathbb{R})\cap \mathbb{Z}\neq\emptyset$.\\
 Moreover, all functions satisfying $(1.1)$ and (i) are of the form $f(x)=\{ x-c\}$ for all (constant) real numbers $c$.
 \end{cor}
 \begin{proof}
This is also a direct corollary of the above results,
considering $g(x):=f^*(x)=x-f(x)$ instead of $f$ (note that $g^*=(f^*)^*=f$).
\end{proof}
\begin{cor}
For $A\subseteq \mathbb{R}$, the following conditions are equivalent:\\
$($a$)$ $\mathbb{R}=A\dot{+}[0,1)$;\\
$($b$)$ $A=c+\mathbb{Z}$, for some $c\in \mathbb{R}$;\\
$($c$)$ $\mathbb{R}=\mathbb{Z}\dot{+}[a_0,a_0+1)$, for some $a_0\in A$, and $A=a_0+\mathbb{Z}$.\\
Hence, $\mathbb{R}=A\dot{+}[0,1)$ and $A\cap \mathbb{Z}\neq\emptyset$ if and only if  $A=\mathbb{Z}$.
 \end{cor}
 \begin{proof}
Note that $\mathbb{R}=A\dot{+}[0,1)$ if and only if $A$ is the range of a decomposer function $f$
with $f^*(\mathbb{R})=[0,1)$. Also, (b) implies $a_0:=c\in A$ and $\mathbb{R}=A\dot{+}[0,1)=
(a_0+\mathbb{Z})\dot{+}[0,1)=\mathbb{Z}\dot{+}[a_0,a_0+1)$.
\end{proof}
The last equivalent condition of the above corollary induces another class of solutions of $(1.1)$.
\begin{cor}
 The general solution of the equation $(1.1)$ with the condition $f^*(\mathbb{R})=[-f(0),-f(0)+1)$ is
  $f(x)=\lfloor x\rfloor+c$ for all (constant) real numbers $c$.
 \end{cor}
 \begin{proof}
All functions of the form $f(x)=\lfloor x\rfloor+c$ satisfy the conditions. Conversely if
the conditions hold, then the function $F(x):=f(x)-f(0)$ satisfies (1), (2) and (3) of Theorem 2.1,
and putting $c=f(0)$ we obtain the result.
\end{proof}
\begin{rem}
It may make the reader wonder of what form the answers will be if we consider one of the conditions
 $f^*(\mathbb{R})=\mathbb{Z}$ or $f(\mathbb{R})=\mathbb{Z}$ for $(1.1)$, and whether the conditions give those special functions?\\
By Theorem 3.6 of \cite{MH3}, under the condition $f^*(\mathbb{R})=\mathbb{Z}$ (resp. $f(\mathbb{R})=\mathbb{Z}$)
 equation $(1.2)$ (resp. $(1.1)$) is equivalent to $f(x+1)=f(x)$ (resp. $f(x+1)=f(x)+1$) which means $f$ is 1-periodic
 (resp. 1-coperiodic). Hence, by using the characterization of periodic (resp. coperiodic)  functions in groups in the same paper
 (Theorem 4.1 and Remark 4.2 of \cite{MH3}), such functions would be of the form $f(x)=\mu(\{x\})$ (resp. $f(x)=\lfloor x\rfloor+\mu(\{x\})$),
 for all functions $\mu:[0,1)\rightarrow \mathbb{R}$. Therefore, we have uncountably many solutions for $(1.1)$
 under the condition  $f^*(\mathbb{R})=\mathbb{Z}$ or $f(\mathbb{R})=\mathbb{Z}$. It is worth noting that for the co-periodic case,
 the ceiling function $\lceil x\rceil$ is also obtained from the general from by putting
 \begin{equation*}
\mu(x) =\left\{\begin{array}{l}
0  \quad : \quad  \mbox{if } x=0 \\
1  \quad : \quad \mbox{if } 0<x<1
\end{array} \right.
\end{equation*}
\end{rem}

\section{Related topics and results}
Since decomposer functions have close relationships to the factorization of groups by two subsets, the achievements in the
previous section (for the real case) can have several related results. Also, one can extend them for many other special functions.
\\
Since $\mathbb{R}=A\dot{+} B$ if and only if $\mathbb{R}=(\lambda A)\dot{+}(\lambda B)$,
for every $\lambda\neq 0$, one may think that $P_{\lambda A}=\lambda P_{A}$. But,
an observation shows that $P_{\lambda A}(x)=\lambda P_{A}(\frac{x}{\lambda})$,
for all $x$. Because
$$P_{\lambda A}(x)+P_{\lambda B}(x)=x=\lambda P_{A}(\frac{x}{\lambda})=\lambda P_{B}(\frac{x}{\lambda}),$$
and $P_{\lambda A}(x),\lambda P_{A}(\frac{x}{\lambda})\in \lambda A$.
Hence, we arrive at the next useful lemma.
\begin{lem} Let $a,b,c,d,e$ and $\lambda$ be arbitrary constant real numbers with $\lambda b\neq 0$, and $A,B$ two subsets.\\
$($a$)$ If $\mathbb{R}=A\dot{+} B$, then $\mathbb{R}=(\lambda A+d)\dot{+}(\lambda B+e)$, and vice versa. Moreover,
$P_{\lambda A+d}(x)=\lambda P_{A}(\frac{x}{\lambda})$ and $P_{\lambda B+e}(x)=\lambda P_{B}(\frac{x}{\lambda})$,
for all $x$.
\\
$($b$)$ If $f$ is a real decomposer function, then all functions of the form
$$bf(\frac{x}{b}+a)+c\; , \; x-bf(\frac{x}{b}+a)+c$$
are decomposer.
In particular, the functions $f^*(x)$, $-f(-x)$, $f(x+a)$, $f(x)+c$, and $f_b(x):=bf(\frac{x}{b})$
are decomposer functions.
\end{lem}
\begin{proof}
This is easy to prove.
\end{proof}
The above lemma lets us extend the results for some real decomposer functions and related factorizations.
For example, the $b$-floor ($b$-integer part) function $\lfloor x\rfloor_b:=b\lfloor\frac{x}{b}\rfloor$,
$b$-ceiling function $\lceil x\rceil_b:=b\lceil\frac{x}{b}\rceil$, and
$b$-fractional ($b$-decimal) part function $\{ x\}_b=b\{\frac{x}{b}\}$
are decomposer functions with the ranges $\lfloor\mathbb{R}\rfloor_b=b\mathbb{Z}=\langle b\rangle=\lceil\mathbb{R}\rceil_b$, and $\{\mathbb{R}\}_b=\mathbb{R}_b:=b[0,1)
 =\{bd|0\leq d<1\}$.
The $b$-part of a real number $a$ was introduced and studied (with the notations $[a]_b$ and $(a)_b$) by the author. In \cite{MH2,MH3,MH4}
many interesting number theoretic, algebraic, and functional properties of the $b$-parts have been studied.
For example, if $b>0$, then $\{a\}_b$ is the same as the remainder of the generalized division of
$a$ by $b$ (see \cite{MH2}). One important result is the finite $b$-representation of real numbers for
every fixed real number $b\neq0,\pm1$. Also, $(\mathbb{R},+_b)$ is a semigroup
where $x+_by:=\{x+y\}_b$ for every $x,y\in \mathbb{R}$, and
$(\mathbb{R}_b,+_b)$ is the largest subgroup of it (an ideal subgroup). Indeed, $(\mathbb{R}_b,+_b)$ is a grouplike
(introduced in 2013) and $(\mathbb{R}_b,+_b)$ is a type of $f$-representatives groups (and related to $f$-grouplikes).
One can find more information in \cite{MH2,MH3,MH4}.
\\
As another result in this regard, we characterize all interval factors of the additive group of real numbers.
\begin{theorem}
The only real intervals that are factors of the additive group of real numbers are the bounded half-open intervals
$($i.e., intervals of the form $[\alpha,\beta)$, $(\alpha,\beta]$$)$ and $(-\infty,+\infty)$.
\end{theorem}
\begin{proof}
First, we show that $[0,1]$ and $[0,1)$ are not factors of $\mathbb{R}$.
If $[0,1]$ is a factor, then $\mathbb{R}=A\dot{+}[0,1]$, for some subset $A$ containing $0$.
Thus $2\in [a_0,a_0+1]$ for some $a_0\in A$ with $a_0> 1$. Since $\frac{a_0+1}{2}\in [a_1,a_1+1]$,
for some $a_1\in A$, then $[a_1,a_1+1]$ intersects either $[0,1]$ or $[a_0,a_0+1]$, that is a contradiction.
Also, if $\mathbb{R}=A\dot{+}(0,1)$ for some subset $A$ containing $0$, then
$0\in (a_0,a_0+1)$, for some $a_0\in A$, and so $(0,1)\cap (a_0,a_0+1)\neq\emptyset$ that is again a contradiction.\\
Now, fix real numbers $\alpha<\beta$.
If $f$ is a real decomposer function with the range $[\alpha,\beta]$ (resp. $(\alpha,\beta)$), then the function
$f_{\frac{1}{\beta-\alpha}}(x)-\frac{\alpha}{\beta-\alpha}$ with the range
$[0,1]$ (resp. $(0,1)$) is so . Therefore, no factor of these forms
exist.\\
Also, the decomposer maps
$$
f(x):=\{ x\}_{\beta-\alpha}+\alpha\; , \; g(x):=x-\lceil x\rceil_{\beta-\alpha}+\alpha
$$
with ranges  $[\alpha,\beta)$ and $(\alpha,\beta]$ require that such intervals are factors.
Finally, if $I$ is an interval other than the mentioned intervals, then $I-I=\mathbb{R}$, and this completes the  proof.
\end{proof}
\begin{rem}
Similar to the uniqueness conditions for the floor, ceiling, and fractional functions, one can drive the general solutions of $(1.1)$ 
with the condition that
  $f(\mathbb{R})$ or $f^*(\mathbb{R})$ equals to either $[\alpha,\beta)$ or $(\alpha,\beta]$, and then one
  can obtain related characterizations.
For example, for the $b$-parts functions we have:
the general solution of the equation $(1.1)$ with the condition $f^*(\mathbb{R})=\mathbb{R}_b$ is
  $f(x)=\lfloor x-c\rfloor_b+c$æ for all $($constant$)$ real numbers $c$. Moreover, the $b$-floor function
  is the only function
 satisfying $(1.1)$ with the conditions:\\
 $ (1)$ $f^*(\mathbb{R})=\mathbb{R}_b$;\\
 $(2)$ $f(\mathbb{R})\cap \langle b\rangle\neq\emptyset$.
\end{rem}

One of the most important equations regarding the floor and ceiling functions is $f(x+1)=f(x)+1$ (the 1-coperiodic functions, see Remark 2.2).
 The necessary condition for $f$ to be 1-coperiodic is that its domain and image (range) both are 1-periodic. We call a subset $A$
 to be 1-periodic
 if $1+A=A$. The unique direct representation (and general form)  of such subsets is $A=D\dot{+} \mathbb{Z}$ where $D\subseteq [0,1)$
 (see \cite{MH5} for a general discussion in groups and semigroups). Therefore, if  $f(x+1)=f(x)+1$, then $\mathbb{Z}\subseteq Im(f)$
 if and only if  $\mathbb{Z}\cap Im(f)\neq\emptyset$. Also, if $f$ is 1-coperiodic and $Im(f)\subseteq \mathbb{Z}$
 (equivalently $Im(f)=\mathbb{Z}$), then $f$ is a decomposer function. In fact, we have $f(x+k)=f(x)+k$, for all $k\in \mathbb{Z}$, and so
 $$f(f^*(x)+f(y))=f(x+f(y)-f(x))=f(x)+f(y)-f(x)=f(y),$$
 for all $x,y$ in the related domain.
  Hence, if $H$ is a subgroup of $(\mathbb{R},+)$ containing 1, and $f:H\rightarrow \mathbb{Z}$ is 1-coperiodic,
  then $f:H\rightarrow H$ is a decomposer function and $H=\mathbb{Z}\dot{+}f^*(H)$. Note that here $f^*(H)$ is a transversal
  of $\mathbb{Z}$ in $H$ and perhaps it has infinitely many choices.
   For the case $H=\mathbb{R}$  there are uncountably many choices for $f^*(H)$, and for $H=\mathbb{Z}$, $f^*(H)$ is a singleton.
   But for the case $H=\mathbb{Q}$ we obtain an interesting result by using Theorem 8
  from \cite{Eis} and the above explanations.
  \begin{lem}
  If $\mathbb{Q}=\mathbb{Z}\dot{+}A$ and $P_A(\frac{P_A(x)}{n})=P_A(\frac{x}{n})$, for all $x\in \mathbb{Q}$ and $n\in \mathbb{Z}_+$,
  then either $A=[0,1)\cap \mathbb{Q}$ or $A=(-1,0]\cap \mathbb{Q}$ $($and vice versa$)$.
  \end{lem}
  \begin{proof}
 Let  $\mathbb{Q}=\mathbb{Z}\dot{+}A$ and put $f:=P_\mathbb{Z}$ (the first projection related to this factorization).
 Then,
 $$
 f^*(x)+(f(x)+1)=x+1=f^*(x+1)+f(x+1),
 $$
 hence $f(x+1)=f(x)+1$ (this is also concluded from Theorem 3.6 of \cite{MH3}),
 since $f(x)+1\in \mathbb{Z}$. Thus, $f$ satisfies the conditions of Theorem 8 of \cite{Eis}, and so
 either $f=\lfloor\;\rfloor\arrowvert_\mathbb{Q}$ or $f=\lceil\;\rceil\arrowvert_\mathbb{Q}$. Therefore, either
 $A=f^*(\mathbb{Q})=[0,1)\cap \mathbb{Q}$ or $A=f^*(\mathbb{Q})=(-1,0]\cap \mathbb{Q}$.
  \end{proof}
\section*{Acknowledgement}
 The author is grateful for the valuable comments of Ali Mohammad Karparvar, who helped to improve the quality of the paper.


\begin{thebibliography}{7}
\bibitem{Art}
E. Artin, {\it The Gamma Function}, Holt Rhinehart \& Wilson, New York,
1964; transl. by M. Butler from Einfuhrung un der Theorie der
Gamma fonktion, Teubner, Leipzig, 1931.
\bibitem{Eis}
P. Eisele and K. P. Hadeler, {\it Game of Cards, Dynamical Systems, and a
Characterization of the Floor and Ceiling Functions}, Amer. Math. Monthly, 97:6 (1990), 466-477.
\bibitem{MH1}
M.H. Hooshmand and H. Kamarul Haili, {\it Decomposer and
Associative Functional Equations}, {Indag. Mathem.}, N.S., Vol.18,
No. 4 (2007),  539-554.
\bibitem{MH2}
M.H.Hooshmand, {\it $b$-Parts and Finite $b$-Representation of Real Numbers},
Notes Number Theory Discrete Math., Vol. 19, No. 4 (2013), 4-15.
\bibitem{MH3}
M.H. Hooshmand, {\it Parter, Periodic and Coperiodic Functions on Groups and their Characterization},
J. Math. Ext., Vol. 7, No. 2 (2013),  1-13.
\bibitem{MH4}
M.H. Hooshmand, {\it Grouplikes}, Bull. Iran. Math. Soc., Vol. 39,
No. 1 (2013),  65-86.
\bibitem{MH5}
M.H. Hooshmand, {\it Upper and Lower Periodic Subsets of Semigroups},
Alg. Colloq. 18:3(2011), 447-460.
\bibitem{MH6}
M.H. Hooshmand, {\it Ultra Power and Ultra Exponential Functions},
Integral Transforms Spec. Funct., Vol. 17, No. 8 (2006),
549-558.
\bibitem{Rub}
C.D. Ruby, {\it Floor and Ceiling Functions}, International Book Market Service Limited, 2012.
\end{thebibliography}
\end{document}